\documentclass[letter,11pt]{amsart}
\usepackage{amssymb} 

\usepackage{hyperref}
\def\thetitle{{Right-angled Artin groups in the $C^\infty$ diffeomorphism group of the real line}}
\hypersetup{
    pdftitle=   \thetitle,
   pdfauthor=  {Hyungryul Baik, Sang-hyun Kim, and Thomas Koberda}
}

\newtheorem{thm}{Theorem}
\newtheorem{lem}[thm]{Lemma}
\newtheorem{cor}[thm]{Corollary}

\theoremstyle{remark}

\theoremstyle{definition}

\newcommand\abs[1]{\lvert #1\rvert}

\newcommand\co{\colon}

\newcommand\R{\mathbb{R}}
\newcommand\Z{\mathbb{Z}}

\newcommand\supp{\operatorname{supp}}

\newcommand\Mod{\operatorname{Mod}}
\newcommand\Homeo{\operatorname{Homeo}}
\newcommand\Fix{\operatorname{Fix}}
\newcommand\DiffR{\operatorname{Diff}^\infty_+(\mathbb{R})}

\newcommand\DiffI{\operatorname{Diff}^\infty_+(I,\partial I)}
\newcommand\DiffS{\operatorname{Diff}^\infty_+(\mathbb{S}^1)}
\newcommand\DiffInf{\operatorname{Diff}^\infty_+([0,\infty),\{0\})}

\newcommand\Diff{\operatorname{Diff}^\infty_+}
\newcommand\PLR{\operatorname{PL}_+(\mathbb{R})}
\newcommand\PLI{\operatorname{PL}_+(I)}

\begin{document}

\title[RAAGs in $\Diff(\R)$]\thetitle

\author{Hyungryul Baik}
\address{Department of Mathematics, Cornell University, 111 Malott Hall, Ithaca, NY 14853, USA}
\email{hb278@cornell.edu}

\author{Sang-hyun Kim}
\address{Department of Mathematical Sciences, Seoul National Univeristy, Seoul 151-747, Republic of Korea}
\email{s.kim@snu.ac.kr}

\author{Thomas Koberda}
\address{Department of Mathematics, Yale University, 20 Hillhouse Ave, New Haven, CT 06520, USA}
\email{thomas.koberda@gmail.com}
\date{\today}
\begin{abstract}
We prove that every right-angled Artin group embeds into the $C^{\infty}$ diffeomorphism group of the real line.
As a corollary, we show every limit group, and more generally every countable residually RAAG group,
embeds into  the $C^{\infty}$ diffeomorphism group of the real line.
\end{abstract}

\maketitle


\section{Introduction}
The \emph{right-angled Artin group} on a finite simplicial graph $\Gamma$ is the following group presentation:
 \[A(\Gamma)=\langle V(\Gamma)\mid [v_i,v_j]=1 \textrm{ if and only if } \{v_i,v_j\}\in E(\Gamma)\rangle.\]  Here, $V(\Gamma)$ and $E(\Gamma)$ denote the vertex set and the edge set of $\Gamma$, respectively.

For a smooth oriented manifold $X$, we let $\Diff(X)$ denote the group of orientation preserving  $C^{\infty}$ diffeomorphisms on $X$. 
A group $G$ is said to \emph{embed} into another group $H$ if there is an injective group homormophism $G\to H$. 
Our main result is the following.




\begin{thm}\label{t:main}
Every right-angled Artin group embeds into $\DiffR$.
\end{thm}

Recall that a finitely generated group $G$ is a \emph{limit group} (or a \emph{fully residually free group}) if for each finite set $F\subset G$, there exists a homomorphism $\phi_F$ from $G$  to a nonabelian free group such that $\phi_F$ is an injection when restricted to $F$. The class of limit groups fits into a larger class of groups, which we call \emph{residually RAAG groups}.  A group $G$ is in this class if for each $1\neq g\in G$, there exists a graph $\Gamma_g$ and a homomorphism $\phi_g\colon G\to A(\Gamma_g)$ such that $\phi_g(g)\ne1$.  
Since the class of right-angled Artin groups is closed under taking finite direct products, we could replace $g$ with an arbitrary finite subset of $G$. 
Theorem~\ref{t:main} can be strengthened as follows.

\begin{cor}\label{c:limit}
Every countable residually RAAG group embeds into 
$\DiffR$.
\end{cor}

A similar argument to the proof of Theorem~\ref{t:main} also applies to the group $\PLR$ of orientation preserving piecewise-linear homeomorphisms of $\R$.

\begin{cor}\label{c:PL}
Every countable residually RAAG group embeds into $\PLR$.
\end{cor}

Our construction requires infinitely many domains of linearity in $\R$, so that we must take infinite subdivisions of $\R$ in order to get residually RAAG groups inside of $\PLR$.  The reader may compare Corollary \ref{c:PL} with the work of Brin and Squier~\cite{BS1985}, which shows that the group $\operatorname{PLF}(\R)$ of piecewise-linear homeomorphisms of $\R$ with finite subdivisions contains no nonabelian free subgroup.

Recall that a group $G$ is called \emph{virtually special} if a finite index subgroup $H\le G$ acts properly and cocompactly on a CAT(0) cube complex $X$ such that the quotient $H\backslash X$ avoids certain pathologies of its half--planes (see \cite{HW2008}).  A consequence of such an action is that the subgroup $H$ embeds in a right-angled Artin group.

\begin{cor}\label{c:vsp}
Let $G$ be a group which is virtually special.  Then there is a finite index subgroup $H\le G$ which embeds into $\DiffR$ and also into $\PLR$.
\end{cor}

Examples of virtually special groups include fundamental groups of closed surfaces and finite volume hyperbolic $3$--manifold groups~\cite{Agol2013,Wise2012}.  Combining the work of Bergeron--Haglund--Wise~\cite{BHW2011} and Bergeron--Wise~\cite{BW2012}, we have that there are virtually special closed hyperbolic manifolds in all dimensions.

The finitely presented subgroups of diffeomorphism groups are generally very complicated.  In~\cite{Bridson2012}, Bridson used a virtually special version of the Rips machine to produce finitely presented subgroups of right-angled Artin groups with exotic algorithmic properties. An arbitrary class of groups which contains sufficiently complicated right-angled Artin groups thereby also has finitely presented subgroups with exotic algorithmic properties (cf.~\cite{KK2014a}):

\begin{cor}\label{c:algorithm}
Suppose $G=\DiffR$ or $G=\PLR$.  Then there is a finitely presented subgroup $H\le G$ such that the conjugacy problem in $H$ is unsolvable.  Furthermore, the isomorphism problem for the class of finitely presented subgroups of $G$ is unsolvable.
\end{cor}

The reader may contrast Corollary \ref{c:algorithm} with Thompson's groups $F$, $T$, and $V$, in which the conjugacy problem is generally solvable~\cite{BF2014}. The group 
$T$ can be embedded into 
 $\DiffS$~\cite{GS1987}.

The group $H$ in Corollary \ref{c:algorithm} is not conjugacy separable, since conjugacy separable groups have solvable conjugacy problems.  In~\cite{FF2001}, Farb and Franks showed that the group of real analytic diffeomorphisms of $\R$ contains non--solvable Baumslag--Solitar groups, which are not even residually finite.


\subsection{Notes and references}
It is well-known that right-angled Artin groups embed in $\Homeo(\R)$, as follows from the fact that right-angled Artin groups admit left--invariant orders.  Similarly, right-angled Artin groups embed in $\Homeo(\mathbb{S}^1)$ because they admit left--invariant cyclic orderings. Alternatively, one can embed an arbitrary right-angled Artin group into the mapping class group $\Mod(S)$ of a compact surface with one boundary, and then embed $\Mod(S)$ into $\Homeo(\mathbb{S}^1)$. As noted in~\cite[p.47]{Farb2006}, it is generally difficult to smoothen these embeddings.  The reader may consult \cite{Navas2011} for a general discussion of the relationship between linear and cyclic orderings of groups and embeddings into $\Homeo(\R)$ and $\Homeo(\mathbb{S}^1)$.

Farb and Franks~\cite{FF2003} proved that every residually torsion--free nilpotent group embeds in the group of $C^1$ diffeomorphisms of the interval and of the circle.  This implies that residually RAAG groups embed in the group of $C^1$ diffeomorphisms of both the interval and of the circle.  Their construction does not allow for twice differentiable diffeomorphisms.  In fact, Plante and Thurston~\cite{PT1976} showed that nilpotent groups of $C^2$ diffeomorphisms of the interval or of the circle are abelian.  In the same vein, Farb and Franks~\cite{FF2003} show that every nilpotent subgroup of $C^2$ diffeomorphisms of $\R$ is metabelian.

For the case when the dimension is two, Calegari and Rolfsen proved that every right-angled Artin group embeds into the piecewise linear homeomorphism group of a square fixing the boundary~\cite{CR2014}.
M.~Kapovich showed that every right-angled Artin group embeds into the symplectomorphism group of the sphere~\cite{Kapovich2012}. The second and the third named authors refined this result by embedding every right-angled Artin group into the symplectomorphism groups of the disk and of the sphere by quasi-isometric group embeddings~\cite{KK2014a}. 

\section{Acknowledgements}
The authors thank B. Farb, A.~Navas and D. Rolfsen for helpful comments.  
The authors are especially grateful to J. Bowden for critical comments regarding a previous version of the paper.
The second named author was supported by National Research Foundation of Korean Government (NRF-2013R1A1A1058646) and Samsung Science and Technology Foundation (SSTF-BA1301-06).
The third named author is partially supported by NSF grant DMS-1203964.

\section{Building an injective homomorphism}
Throughout this section, we let $\Gamma$ be a finite graph.
Consider an element $g$ of $A(\Gamma)$.
A \emph{reduced word representing $g$} means a minimal length word in the standard generating set $V(\Gamma)\cup V(\Gamma)^{-1}$ representing $g$.
The \emph{support of $g$} is the set of vertices $v$ of $\Gamma$ such that $v$ or $v^{-1}$ appears in a reduced word representing $g$. We denote the support of $g$ by $\supp(g)$.
We say $g$ is a \emph{clique word} if every pair of vertices in $\supp(g)$ are adjacent in $\Gamma$. 
A \emph{clique word decomposition for $g$} is the concatenation $w_k\cdots w_1$ of clique words $w_1, w_2,\ldots,w_k$ such that the concatenation is still reduced and represents $g$ in $A(\Gamma)$. Since vertices themselves are clique words, every element in $A(\Gamma)$ has a clique word decomposition.
A clique word decomposition $w_k\cdots w_1$ is \emph{left--greedy} if for each $i<k$ and for each $v\in\supp(w_i)$, there exists a vertex $v'\in\supp(w_{i+1})$ such that $[v,v']\ne1$.  The left--greedy clique word decomposition can be compared to the left--greedy normal form used in \cite{Koberda2012}.

\begin{lem}\label{l:leftgreedy}
Every element of $A(\Gamma)$ admits a left--greedy clique word decomposition.
\end{lem}
\begin{proof}
Fix an element $g\in A(\Gamma)$.
Let us define the \emph{complexity} of a clique word decomposition $w_k w_{k-1} \cdots w_1$ for $g$ as the $k$-tuple consisting of the word lengths $(\abs{w_k},\abs{w_{k-1}},\ldots,\abs{w_1})$.
In the lexicographical order, this complexity is bounded above by $(\abs{g})$.
Let us assume $w_k w_{k-1}\cdots w_1$ is the maximal clique word decomposition for $g$ in this order.
If $w_k\cdots w_1$ is not left--greedy, then there exists $i$ and $v\in\supp(w_{i-1})$ be such that $v\cup\supp(w_i)$ spans a complete subgraph of $\Gamma$.  Then we can move an occurrence of $v$ or $v^{-1}$ in $w_{i-1}$ to $w_i$, i.e. we \emph{slide to the left}.  This is a contradiction to the maximality.
\end{proof}

Let $X$ be a smooth oriented manifold.
If $f\in\Diff(X)$, then we write the fixed point set of $f$ as $\Fix(f)$. The closure of $X \setminus \Fix(f)$ will be denoted as $\supp(f)$.
For $f_1,f_2,\ldots,f_n\in\Diff(X)$, let us define the \emph{disjointness graph} $\Lambda$ by
$V(\Lambda)=\{v_1,v_2,\ldots,v_n\}$ and 
\[
E(\Lambda) = \{\{v_i,v_j\} \co i\ne j\textrm{ and }\supp(f_i)\cap\supp(f_j)=\varnothing\}.\]
Note that two self-maps with disjoint supports commute. Hence, we have a group homomorphism $\phi\co A(\Lambda) \to \Diff(X)$ satisfying $\phi(v_i)=f_i$.
Often, it is not an obvious task at all to decide whether or not such a map $\phi$ is injective; see~\cite{CW2007,Kapovich2012,KK2014a,CR2014} for related work on diffeomorphism groups and~\cite{CP2001,CW2004,Koberda2012,CLM2012} on mapping class groups.

Let $H$ be a group. Another group $G$ is \emph{residually} $H$ if for each nontrivial element $g$ in $G$ there exists a group homomorphism $\phi_g\co G\to H$ such that $\phi_g(g)\ne1$.
For a smooth oriented manifold $X$, we let $\Diff(X,\partial X)$ denote the group of orientation preserving $C^{\infty}$ diffeomorphisms on $X$ which restrict to the identity near $\partial X$. 
The crucial fact we need is the following:
\begin{lem}\label{l:eachword}
The group $A(\Gamma)$ is residually $\DiffI$.
\end{lem}

\begin{proof}
Let us fix an element $g$ in $A(\Gamma)$.
We let $w_k \cdots w_1$ be a left--greedy clique decomposition representing $g$.
We will construct a group homomorphism 
\[\phi_g\co A(\Gamma) \to \DiffI\] such that $\phi_g(g)\ne1$.
We can inductively choose a (possibly redundant) sequence 
\[
v_1\in\supp(w_1),
v_2\in\supp(w_2),\ldots, v_k\in\supp(w_k)\] such that $[v_i,v_{i+1}]\ne1$ for each $i=1,2,\ldots,k-1$.
There exists $\sigma_i\in\{-1,1\}$ and $n_i>0$ such that $v_i^{\sigma_i n_i}$ is the highest power of $v_i$ in $w_i$. This means that $v_i\not\in\supp(w_i v_i^{-\sigma_i n_i})$.

We choose  $\rho\in\DiffR$ such that $\rho(1/4)=5/4$
and $\rho(x) = x$ for $x\le 0$ or $x\ge 3/2$. Put $\rho_i(x) = \rho(x-i)+i$ and $I_i = [i,i+3/2]$, so that $\supp\rho_i\subseteq I_i$. Note that $I_i\cap I_j =\emptyset$ for $\abs{i-j}>1$.

For each $v\in V(\Gamma)$, we define
\[
\psi_g(v) = \prod_{v_j=v} \rho_j^{\sigma_j}.
\]
This means that if $\{j \co v_j = v\} = \{j_1<j_2<\cdots<j_n\}$ then
$\psi_g(v) = \rho_{j_1}^{\sigma_1}\rho_{j_2}^{\sigma_2}\cdots \rho_{j_n}^{\sigma_n}$.
We use the convention that the empty multiplication is trivial.
Let us write $J_v = \cup_{v_j = v} I_j\supseteq \supp \psi_g(v)$.
Note that if $v_i=v=v_j$ and $i\ne j$, then $I_i\cap I_j=\emptyset$.
For each $\{u,v\}\in E(\Gamma)$, 
the choice of $v_1,v_2,\ldots,v_k$ implies that $J_u\cap J_v=\emptyset$.
In other words, $\Gamma$ is a subgraph of the disjointness graph of $\{\psi_g(v)\co v\in V(\Gamma)\}$.
It follows that $\psi_g$ defines a group homomorphism from $A(\Gamma)$ to $\DiffR$.

Suppose $\ell\in\{1,2,\ldots,k\}$.
For every $u\in\supp w_\ell \setminus\{v_\ell\}$, we have $J_u\cap J_{v_\ell}=\emptyset$.
Since $\ell+1/4 \in I_\ell \subseteq J_{v_\ell}$, 
we see that if $x\in [\ell+1/4,\ell+1/2]$ is an arbitrary point then 
\[\psi_g(w_\ell)(x) = \psi_g(v_\ell^{\sigma_\ell n_\ell})(x)=\rho_\ell^{n_\ell}(x) \in [\ell+5/4,\ell+3/2].\]
It follows that if $y\in [5/4,3/2]$ is arbitrary then 
\[\psi_g(g)(y)=\psi_g(w_k w_{k-1}\cdots w_1)(y) \in [k+5/4,k+3/2].\] 
In particular, $\psi_g(g)$ is not the identity.
By restricting the image of $\psi_g$ onto the interval $[0,k+2]$ and conjugating by a diffeomorphism $[0,k+2]\approx I$, we obtain a desired group homomorphism $\phi_g$.
\end{proof}

We have the following general fact.
\begin{lem}\label{l:residual}
Let $G$ and $H$ be groups.
If $G$ is countable and residually $H$, then $G$ embeds into the countable direct product $\prod_\Z H$.
\end{lem}

\begin{proof}
Define
\[\psi \co G\to\prod_{g\in G}H\]
by $\psi(x)(g)=\phi_{g}(x)$, where $\phi_g$ is as in the definition of residually $H$ group.
\end{proof}

Since $\DiffInf\hookrightarrow\DiffR$, Theorem~\ref{t:main} a trivial consequence of the following:

\begin{thm}\label{t:stronger}
Every right-angled Artin group embeds into $\DiffInf$.
\end{thm}
\begin{proof}
Immediate from 
Lemmas~\ref{l:eachword} and~\ref{l:residual}, as well as the fact that 
\[\prod_\Z \DiffI\hookrightarrow\DiffInf.\qedhere\]
\end{proof}

\begin{proof}[Proof of Corollary~\ref{c:limit}]
Lemma~\ref{l:eachword} implies that a residually RAAG is residually $\DiffI$.
We proceed as Theorem~\ref{t:stronger}.
\end{proof}

\begin{proof}[Proof of Corollary~\ref{c:PL}]
It suffices to show that every right-angled Artin group is residually $\PLI$. For this, 
we follow the proof of Lemma~\ref{l:eachword} by using 
$\rho_0\in \PLR$ defined by
\[\rho_0(x) = 
\begin{cases} x &\mbox{if } x<0 \mbox{ or }x\ge\frac32\\ 
5x & \mbox{if } 0\le x<\frac14\\
(x+6)/5 & \mbox{if } \frac14\le x<\frac32
 \end{cases}
 \]
 instead of $\rho$.
\end{proof}

\bibliographystyle{amsplain}
\bibliography{./ref}

\end{document}